\documentclass[10pt,a4paper]{amsart}
\usepackage[margin=15mm]{geometry}

\usepackage{epsfig}
\usepackage{amssymb,amsmath,amsthm}
\usepackage[all]{xypic}
\usepackage[numbers]{natbib}
\usepackage[colorlinks]{hyperref}
\newtheorem{theorem}{Theorem}[section]

\newtheorem{definition}{Definition}[section]
\newtheorem{example}{Example}[section]

\newtheorem{lemma}{Lemma}[section]

\theoremstyle{remark}
\newtheorem{remark}{Remark}[section]

\newcommand{\x}{\boldsymbol{x}}

\newcommand{\R}{\mathbb{R}}

\newcommand{\p}{\mathbb{P}}
\newcommand{\Span}[1]{\mathrm{Span}\left\langle#1\right\rangle}

\DeclareMathOperator{\dd}{d}%
\newcommand{\CC}{\mathcal{C}}
\newcommand{\Th}{^\textrm{th}}

\title{Geometry of the free--sliding Bernoulli beam}
\author{Giovanni Moreno and Monika Ewa Stypa}

\begin{document}

\maketitle

\begin{abstract}
If a variational problem comes with no boundary conditions prescribed beforehand, and yet these arise as a consequence of  the variation process itself, we speak of a \emph{free boundary values variational problem}.   Such is, for instance, the problem of finding the shortest curve whose endpoints can slide along two prescribed curves. There exists a rigorous geometric way to formulate this sort of problems on smooth manifolds with boundary, which we review here in a   friendly self--contained way. As an application, we study a particular  free boundary values variational problem,  the \emph{free--sliding Bernoulli beam}.\par
This paper is dedicated to the memory of prof. Gennadi Sardanashvily.
\end{abstract}

\tableofcontents

\section{Introduction}

The Euler--Lagrange equations are not the only conditions satisfied by  a solution to a variational problem. In the majority of the cases we do not see these additional conditions simply because they become trivial.  There exist, however, important and physically significant examples, where they are by no means trivial and rather play a key role in the description of the solutions to the variational problem itself.\par
We shall call these conditions \emph{natural boundary conditions}. Our   aim     is to obtain the natural boundary conditions for a second--order one--dimensional variational problem henceforth referred to as the \emph{free--sliding Bernoulli beam}. But  such a particular result will be   framed against   a general coordinate--free background.  The  class of problems where natural boundary conditions arise and are nontrivial is rather vast, and we may  call it the class of \emph{free boundary values variational problems} (see, e.g.,  \cite{MR1368401},  Chapter 2,  Section 4,   and  \cite{MR2004181}, Chapter 7). The      geometric framework  sketchy reviewed  below  is valid in the entire  class.\par
The idea of a free--sliding Bernoulli beam is implicit in many classical treatments of variational calculus (see, e.g., C. Lanczos \cite{MR0431821}, Chapter II, Section 15). However,  the problem itself has never been discussed in details and, in particular, the corresponding  natural boundary conditions have  never been derived before---to the authors' best knowledge. On the theoretical side, a solid geometric framework for free boundary values variational problems can be easily obtained by generalising any of  the various modern homological approaches to variational calculus, which all  follow in spirit the original works by Janet \cite{Janet:LSDP} and Dedeker \cite{Ded53}. The key ingredient is  the so--called \emph{relative homology}, i.e., the de Rham theory adapted to manifolds with boundary (see \cite{MR3458999} and references therein). This is a rather straightforward step, but it may  be hard to grasp for a newcomer, as it requires a minimal toolbox of homological algebra and differential topology. This is why we review it below, in a friendly but rigorous way, especially designed to fit   the main example at hand. %
\subsection{Formulation of the main problem}\label{secFormMainProb}
If $u=u(x)$ and we denote by $p,q,r,s$ the first, second, third and fourth derivative of $u$ with respect to $x$, respectively, then the one--form
\begin{equation}\label{eqLagBernoulli}
\lambda=\left(\kappa\frac{q^2}{2}-\rho u\right)dx
\end{equation}
may be understood as a second--order Lagrangian. It is in fact  a well--known variational principle, discovered around 1750. It allows    to describe   a massive beam which bends  under its own weight and against its internal reaction forces (see, e.g.,  \cite{MR0010851}). Nowadays we speak of a ``Bernoulli beam''. In compliance with the terminology used by both \cite{MR2004181} and \cite{MR0431821}, $\kappa$ is a nonzero constant, and $\rho=\rho(x)$ is a function, representing the elasticity of the beam and the load, respectively.\par
The Euler--Lagrange equations associated with \eqref{eqLagBernoulli} are easily computed, viz.
\begin{equation}\label{eqELprimordiali}
s=\frac{1}{\kappa}\rho\, .
\end{equation}
The  general solution 
\begin{equation}\label{eqELsemplice}
u(x)=u_0(x)+c_3x^3+c_2x^2+c_1x^1+c_0
\end{equation}
to \eqref{eqELprimordiali} depends on four integration constants. In the most typical situations, the circumstances provide the correct number of boundary conditions, so that  the solution to \eqref{eqELprimordiali} becomes unique (see, e.g., \cite{MR2004181}, Section 7.1 and \cite{MR0431821}, Chapter II, Section 15). And even when the circumstances \emph{do not} prescribe enough boundary conditions, more   arise as a consequence of the variation problem itself.  An important example of such   phenomenona is provided by the so--called \emph{cantilever beam} (see, e.g.,  \cite{MR2004181}, Section 7.1, Case II).\par 
The free boundary values variational problem  studied in this paper can be described as follows. The endpoints of the Bernoulli beam $L$ can slide freely, that is frictionlessly, along a prescribed curve $\Gamma$  (see Figure  \ref{fig1}). The curve  $\Gamma$ is   assumed henceforth to be the boundary of a connected domain $E\subset\R^2$, i.e., $\Gamma=\partial E$. Our goal is to show that the four--parametric family of solutions \eqref{eqELsemplice} reduces to a generically two--parametric  family of solutions as a consequence of the fact that we have forced the endpoints of $L$ to lie on $\Gamma$. This program is carried out in the last Section \ref{secSECONDA}.\par The problem of how the geometry of $\Gamma$ can further reduce the multitude of the solutions---possibly all the way down to uniqueness, or even to non--existence---is an interesting one, but not touched upon here. \par
The reader may have noticed that, as opposed to the ``classical'' treatment of the Bernoulli beam, which is described in terms of a function $u$, here we deal with one--dimensional submanifolds (i.e., curves) of the two--dimensional manifold with boundary $E$. This is precisely why we need a coordinate--free approach to free boundary values variational problems on manifolds with boundary. The following example, already worked out in \cite{MR2757930}, clarifies this important step. 

\begin{figure}[h]
\epsfig{file=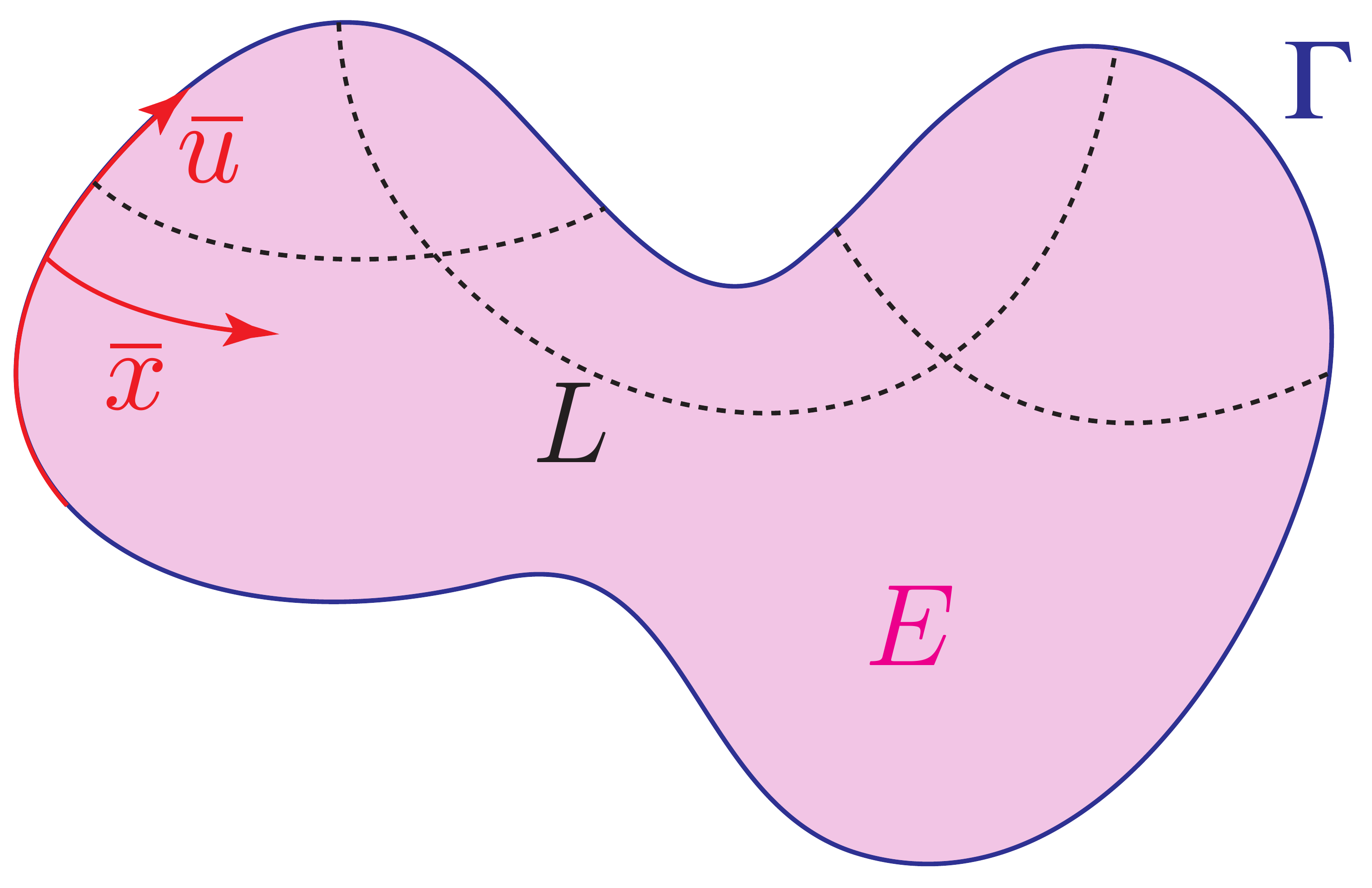,width=0.5\textwidth}\caption{Two arbitrary points of the closed curve $\Gamma$ can be joined by a beam $L\subset E$. If such points are assumed to be freely sliding along $\Gamma$, then the corresponding variational problem is a \emph{free boundary values variational problem}.\label{fig1}}
\end{figure}

 \subsection{A motivating example}\label{secMot}
Let $\R^2$ be equipped with the Euclidean metric, and let $\gamma:[a,b]\to\R^2$ be a   curve. Then the length of $\gamma$ is given by 
\begin{equation}\label{eqLenFun}
\ell(\gamma)=\int_a^b \|\dot{\gamma}(t)\| dt\, .
\end{equation}
It is a basic undergraduate exercise to show that $\gamma$ is a critical point for $\ell$ if and only if $\gamma$ is a straight line. In particular, this means that the set of solutions to the Euler--Lagrange equation associated with \eqref{eqLenFun} is a two--parametric family.\par
Straight lines are the correct answer to   the question ``what are the shortest lines between two arbitrary points of $\R^2$?"\par
But, what if we are interested in the shortest lines joining two arbitrary points \emph{lying on the boundary $\Gamma=\partial E$ of a connected domain $E\subset\R^2$} instead? Intuitively, since any pair of points of $\Gamma$ is a pair of points of $\R^2$, the Euler--Lagrange equations need to be satisfied, but new equations are needed, and these extra equations must depend on the particular choice of $E$. Obvious.\par
However, formalising this fact in a proper way requires an unexpectedly elaborated framework, even if the idea behind is extremely simple (see \cite{MR2456137,MorenoCauchy,MR2757930}).  The case when $E=[a,b]\times\R^2$ is particularly simple, and deserves to be discussed further. First of all, since we are  interested\footnote{The  examples  really worth studying are those where the boundary $\partial E$ is \emph{not} a solution to the Euler--Lagrange equations.} in lines joining a point of $\{a\}\times\R$ with a point of $\{b\}\times\R$, it suffices to consider curves of the form $(x,u(x))$, i.e., functions on $[a,b]$. Hence,  \eqref{eqLenFun} reads
\begin{equation}\label{eqLenFunLOC}
\ell(u)=\int_a^b \sqrt{1+p^2}dx\, .
\end{equation}
In order to get the Euler--Lagrange equations, we compute the variation
\begin{align}
\frac{\delta\ell}{\delta \overline{u}}(u)&= \left.\frac{\dd  }{\dd \epsilon} \right|_{\epsilon=0}\ell(u+\epsilon \overline{u})\nonumber\\
&=\int_a^b \frac{ p }{\sqrt{1+p^2}}\overline{u}'dx\nonumber\\
&=-\int_a^b \frac{\dd}{\dd x} \left(\frac{ p }{\sqrt{1+p^2}}\right)\overline{u}dx+\left. \frac{ p }{\sqrt{1+p^2}} \overline{u}\right|_a^b   \, .\label{eqELsemplicissima}
\end{align}
Now observe that the boundary $\partial E$ is the ``space of admissible boundary values'' for the unknown function $u\in C^\infty([a,b])$. Hence, we speak of a \emph{free boundary values variational problem} precisely because these values are not prescribed beforehand.  But, if a curve/function $u$ is critical for \eqref{eqLenFunLOC} for \emph{arbitrary} boundary values, then, in particular, it must be critical for \eqref{eqLenFunLOC} for \emph{prescribed}  boundary  values.  In practice, this means that \eqref{eqELsemplicissima} must vanish for all $\overline{u}$ such that $\overline{u}(a)=\overline{u}(b)=0$, i.e., that the second--order Euler--Lagrange equation $q=0$ must be satisfied by $u$. Expectedly, we obtain   all polynomial functions of degree $\leq 1$, which is a two--parametric family (see Figure \ref{fig2} below).\par
Needless to say, polynomial functions of degree $\leq 1$ do not answer the question ``what are the shortest lines joining a point of $\{a\}\times\R$ with a point of $\{b\}\times\R$?''  But the desired functions are among them, that is, the solutions to our problem must, in particular, satisfy the equation $q=0$. It remains to observe that, for functions satisfying $q=0$,  \eqref{eqELsemplicissima} reads
\begin{equation}\label{eqTCsemplicissima}
\frac{\partial\ell}{\partial \overline{u}}(u)=\left. \frac{ p }{\sqrt{1+p^2}} \overline{u}\right|_a^b\, .
\end{equation}
Indeed, the variation \eqref{eqTCsemplicissima} must vanish now for all $\overline{u}$, i.e., we do not need to assume anymore that $u$ and $u+\epsilon \overline{u}$ have the same boundary values for small $\epsilon$. It follows immediately from \eqref{eqTCsemplicissima} that $p(a)=p(b)=0$ and we correctly get only the constant functions. That is, the family of the solution to our free boundary values variational problem is one--parametric, whereas the family of the solutions of the Euler--Lagrange equation $q=0$ alone is two--parametric.\par
Passing to less trivial domains $E$, like a circle or an ellipse (see Figure \ref{fig2} below), we see  that both the quantity and the nature of the solutions to the corresponding free boundary values variational problems  heavily depend on the shape of $E$.  In the case of a  circle, we still get a one--parametric subfamily, but different from the previous example, even topologically. If $E$ is an  ellipse, then we even obtain a discrete family (two elements). \par
These trivial examples should make the reader suspect that there is a deeper theory behind. A theory where the ``new equations''  accompanying the Euler--Lagrange \emph{depend} on $E$. These are precisely the   \emph{natural boundary conditions}.\footnote{In this particular example the natural boundary conditions are usually referred to as the \emph{transversality conditions} (see \cite{MR2004181}, Section 7.3, and \cite{TesiDelCacchio})} Unfortunately, formalising properly this dependency requires a rather heavy jet--theoretical formalism (see \cite{MorenoCauchy}), which would not be in the spirit of this paper. In Section \ref{secPRIMA} below we provide a minimalistic toolbox needed to deal with the main problem (the one described in Section  \ref{secFormMainProb}), by avoiding   general theoretical considerations and focusing on the useful results instead.\par
We close this introduction by   proposing an analogy. The passage from the ``global'' Lagrangian \eqref{eqLenFun} on $\R^2$ to the ``restricted'' Lagrangian \eqref{eqLenFunLOC} on $E$ is morally the same as the passage from   a differential form $\omega$ on a manifold $M$ to its restriction $\omega|_E$ to a submanifold $E\subset M$ of codimension zero.  So, even symbolically, $\omega|_E$ carries a reference to the submanifold $E$, whereas \eqref{eqLenFunLOC} does not. But what really matters is that  the de Rham theory changes completely.\par
More precisely, the notion of \emph{exact forms} is different. A form $\omega\in\Omega^k(M)$ is exact if it is the differential $d\eta$ of a form $\eta\in\Omega^{k-1}(M)$. But the restricted form $\omega|_E $ is exact if it is the differential   of a \emph{relative form}, i.e., an element of $ \Omega^{k-1}(E,\partial E)$. A relative form on $E$ is a form which vanishes on $\partial E$. The corresponding \emph{relative de Rham complex} is indicated by $(\Omega(E,\partial E), d_{\textrm{rel}})$. If $k=\dim E=\dim E$, then  the relative de Rham cohomology  on $E$ has the same $k$--cycles but fewer $k$--boundaries. Since the Lagrangians and the corresponding Euler--Lagrange expressions  can be interpreted as suitable de Rham cohomology classes on jet prolongations of $E$, the passage  from  \eqref{eqLenFun} to  \eqref{eqLenFunLOC} causes, in particular, the Euler--Lagrange expressions to land into a larger cohomology space, which explains the appearence of the natural boundary conditions  (see \cite{MR3458999} for a more exhaustive discussion) next to the classical Euler--Lagrange equations. \par
The profound difference between these two pieces of the same cohomological object should always been kept in mind. The latter are just the restriction of the Euler--Lagrange equations on $M$: as such, no matter which $E$ is chosen, they will look the same in the neihghborhood of any internal point of $E$. The former depend on the choice of $E$:   the same point $p\in M$ may be a boundary point of several  different domains $E$, and for each choice of $E$ the corresponding equations need not to be the same. 

\begin{figure}[h]
\epsfig{file=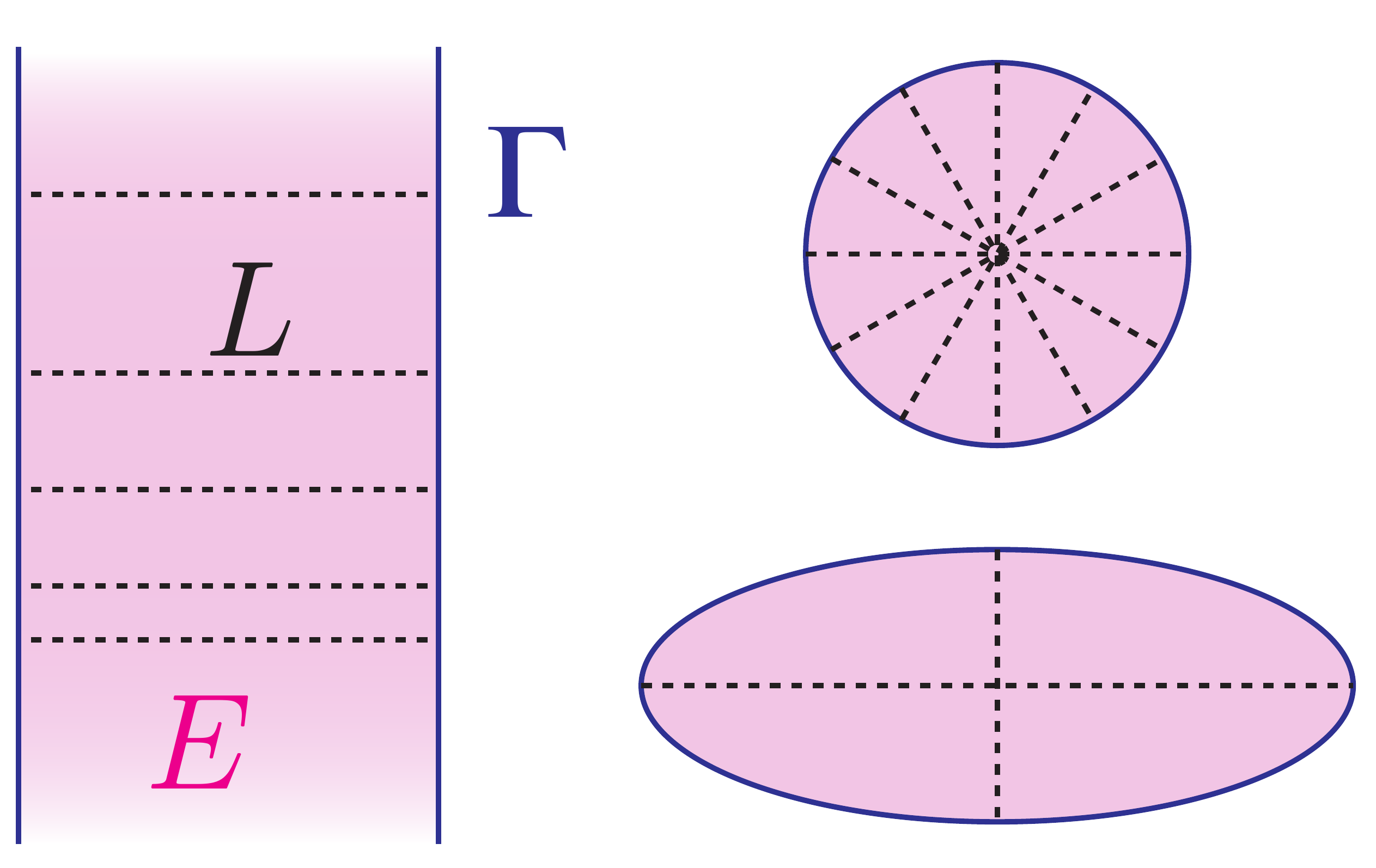,width=0.5\textwidth}\caption{Examples of  free boundary values variational problems.}\label{fig2}
\end{figure}
%
%
%
%
%
\section{The geometry of free boundary values variational problems}\label{secPRIMA}
Let $E$ be a (connected) two--dimensional manifold with (connected) nonempty boundary $\partial E$. In order to speak of a free boundary value variational problem imposed on one--dimensional submanifolds (i.e., curves) in $E$, we need the correct geometric counterpart of several intuitive notions, such as:
\begin{enumerate}
\item admissible curves in $E$,
\item a   Lagrangian on $E$,
\item the Euler--Lagrange equations associated with the Lagrangian,
\item the natural boundary conditions associated with the Lagrangian.
\end{enumerate}
The less familiar entry is surely the last one, which will need a special care. Together with these indispensable gadgets, we shall also need some key property in order to obtain our result, namely:
\begin{enumerate}
\item the locality of the Euler--Lagrange equations,
\item the locality of the natural boundary conditions,
\item the invariance of Euler--Lagrange equations with respect to  diffeomorphisms of $E$,
\item the invariance of natural boundary conditions wth respect to   diffeomorphisms of $E$.
\end{enumerate}
 Below we clarify all these points in the most direct and self--contained way. All will be used in next   Section \ref{secSECONDA} in order to obtain the desired result.

\subsection{Admissible curves and  coordinated patches}\label{subAdmin}
This is the easiest notion. A curve is admissible for a variational problem on $E$ if it is, roughly speaking, transversal to $\partial E$. The reason for that can be easily grasped by thinking at the length Lagrangian \eqref{eqLenFun}: a curve hitting the boundary of $E$ tangentially can never be  length--minimising (see Section \ref{secMot} above). As such, it should be ruled out from the very beginning.
\begin{definition}\label{defAdmCurv}
 A curve $L$ is \emph{admissible} if $\partial L=L\cap\partial E$ and $L$ is nowhere tangent to $\partial E$. The set of admissible curves is denoted by $\mathcal{A}(E)$.
\end{definition}
\begin{example}\label{exEsempioBanale}
If $E=I\times\R$, where $I\subset\R$ is a closed interval, then the admissible curves in $E$ are just the graphs of the smooth functions on $I$ (see Figure \ref{fig2}, leftmost example).
\end{example}
As a manifold with boundary, $E$ possesses two kind of coordinate patches, namely the ordinary (standard, internal) ones, and the boundary ones (like the $(\overline{x},\overline{u})$ patch depicted in Figure \ref{fig1}). A coordinate patch is a pair $(U,\x)$, where $U\subseteq E$ is open and $\x$ is a diffeomorphism between $U$ and $\R^2$ (standard) or $[0,\infty[\times\R$ (boundary).

\subsection{Jet spaces and Lagrangians}
The $k\Th$ jet extension $J^kE$ is a bundle over $E$ whose fibre $J^k_{\x} E$ at $\x\in E$ is defined by
\begin{equation}\label{eqDefJK}
J^k_{\x} E:=\{ L\mid L\textrm{ is a curve in }E\}/\sim^k_{\x}\, ,
\end{equation}
where
\begin{equation*}
L\sim^k_{\x} L'\Longleftrightarrow L\textrm{ is tangent to }L'\textrm{ at }\x\textrm{ with order }k\, .
\end{equation*}
Jet spaces form a natural tower of one--dimensional smooth bundles, usually denoted by
\begin{equation}\label{eqTower}
\cdots\rightarrow J^kE\stackrel{\pi_{k,k-1}}{\longrightarrow} J^{k-1}E\rightarrow\cdots\rightarrow J^1E\rightarrow E\, .
\end{equation}
According to the general theory (see, e.g., \cite{MR989588,MR1670044}), $\pi_{1,0}$ is a smooth $\R\p^1$--bundle, whereas the other are  affine $\R$--bundles.\par
Let $(U,\x)$ be an internal (resp. boundary) coordinate patch, with $\x=(x,u)$. For any function $u=f(x)$ denote by
\begin{equation*}
L_f:=\x^{-1}( \{ (x,u)\mid u=f(x)\, ,\ x\in\R\ (\textrm{resp., }x\in[0,\infty[) \})
\end{equation*}
its   \emph{graph}.  We define the subset
\begin{equation*}
\widetilde{U}:=\{ [L_f]^k_{\x}\mid \x\in U\, ,\ f\in C^\infty(\R)  \ (\textrm{resp., }f\in C^\infty([0,\infty[) )\}\subseteq J^kE\, , 
\end{equation*}
and the map
\begin{align*}
\widetilde{U}\ni [L_f]^k_{\x}\stackrel{\widetilde{\x}}{\longmapsto}& (x,u=f(x), p=f'(x),q=f''(x), r=f'''(x),\\ & s=f^{(iv)}(x),\ldots, u^{(k)}=f^{(k)}(x))
\end{align*}
between $\widetilde{U}$ and $\R^2\times \R^k$ (resp., $[0,\infty[\times\R\times \R^k$). Notice that we have denoted by $(p,q,r,s, u^{(5)}, u^{(6)}, \cdots u^{(k)})$ the coordinates on  $\R^k$. It can be proved that $\widetilde{U}$ is open in $J^kE$.
\begin{definition}
 The pair $(\widetilde{U}, \widetilde{\x})$ is called a ordinary (resp., boundary) \emph{affine coordinate patch} on $J^kE$ (induced by $(U,\x)$).
\end{definition}
When viewed through coordinate patches, the tower of bundles \eqref{eqTower}  is simply
\begin{equation*}
\ldots\longrightarrow E\times\R^k\stackrel{\pi_{k,k-1}}{\longrightarrow} E\times\R^{k-1}\longrightarrow\ldots \, .
\end{equation*}
For any curve $L\subset E$, define its $k\Th$ \emph{jet extension}
\begin{equation*}
L^{(k)}:=\{  [L]_{\x}^k\mid \x\in L\}\subset J^kE\, ,
\end{equation*}
which, as a submanifold of  $J^kE$, is a curve as well (see, e.g., \cite{MR1670044}).  Plainly, not all the curves in $J^kE$ are of the form  $L^{(k)}$, but jet spaces come equipped with a structure which allows precisely to tell one from another. More precisely, fix a  point $\theta:=[L]_{\x}^k\in J^kE$, and observe that the representative $L$ is, of course, not uniquely defined. Such an anbiguity allows us to define the nontrivial subspace
\begin{equation*}
\CC^{(k)}_\theta:=\Span{ T_\theta  L\mid [L]_{\x}^k=\theta}\subset T_\theta J^kE\, .
\end{equation*}
\begin{definition}The distribution $\CC^{(k)}$ given by 
$\theta\longmapsto \CC^{(k)}_\theta$ is called the \emph{contact distribution} on $J^kE$.
\end{definition}
In an affine coordinate patch, $\CC^{(1)}$ is the annihilator of the one--form
\begin{equation}\label{eqContForm1}
du-pdx\, ,
\end{equation}
$\CC^{(2)}$ is the annihilator of the one--form \eqref{eqContForm1}, together with the one--form
\begin{equation}\label{eqContForm2}
dp-qdx\, ,
\end{equation}
$\CC^{(3)}$ is the annihilator of the one--forms \eqref{eqContForm1} and  \eqref{eqContForm2}, together with the one--form
\begin{equation}\label{eqContForm3}
dq-rdx\, ,
\end{equation}
and so on so forth.
\begin{theorem}\label{thTheoremaLie}
 A curve in $J^kE$ is of the form $L^{(k)}$ if and only if it is an integral curve of $\CC^{(k)}$ nondegenerately projecting over $E$.
\end{theorem}
\begin{proof}
 See, e.g., \cite{MR1670044}.
\end{proof}
 By dualising the tower of projections \eqref{eqTower}, one sees that the modules of differential forms on lower--order jet spaces can be embedded into the modules of differential forms on higher--order jet spaces. In particular, $\Omega^1(E)$ can be considered as a $C^\infty(E)$--submodule in any $\Omega^1(J^kE)$. 
\begin{definition}
The $C^\infty(J^kE)$--submodule $\Omega_h^1(J^kE)$  generated by $\Omega^1(E)$ within $\Omega^1(J^kE)$ is called the module of \emph{$k\Th$ order Lagrangians} on $E$.
\end{definition}
\begin{example}\label{exEsempioBanale2}
 In the setting of Example \ref{exEsempioBanale}, the formula \eqref{eqLagBernoulli}   correctly defines    a second--order Lagrangian on $E$.  Indeed, the one--form $dx$ over $M$ is multiplied by such scalars, like $p$ and $q$, which belong to $C^\infty(J^2E)$. 
\end{example}

\subsection{The Euler--Lagrange equation}
The construction of the Euler--Lagrange equation associated to a (in our context, second--order)   Lagrangian $\lambda\in \Omega_h^1(J^2E)$ has been the subject of dozens of works (see, e.g., \cite{Ded53,Janet:LSDP,MR0394755,MR0377987,Takens:SCLVP,MR0501129,MR590637,MR667492,MR739951,MR739952,MR1262598}), each showing the same phenomenon form a different perspective. We only stress here that  the left--hand side of Euler--Lagrange equation $\delta\lambda=0$ associated with $\lambda$ is a particular two--form on $J^4E$, called the \emph{Euler--Lagrange expression}. More precisely,
\begin{equation}\label{eqModulo11}
\delta\lambda\in \Omega_{h,c}^2(J^4E) \, ,
\end{equation}
where $ \Omega_{h,c}^2(J^4E) $ denotes the submodule of $ \Omega^2(J^4E) $ generated by the products of   horizontal one--forms with   contact one--forms.
\begin{example}\label{exEsempioBanale3}
 In the same setting as Examples \ref{exEsempioBanale} and \ref{exEsempioBanale2}, the Euler--Lagrange expression can be written in terms of the so--called Euler--Lagrange derivative of the Lagrangian density. So, formula \eqref{eqELprimordiali} is correctly interpreted as
 \begin{equation}\label{eqELnobile}
\delta\lambda=\left( s-\frac{1}{\kappa}\rho \right)\otimes (du-pdx)\, ,
\end{equation}
where we used the contact form \eqref{eqContForm1}. The two--form \eqref{eqELnobile} is called, in the bicomplex terminology, a form of type $(1,1)$.
\end{example}
A second--order Lagrangian determines an action functional on the space $\mathcal{A}(E)$ of admissible sections, usually denoted by
\begin{equation}\label{eqDefAcFun}
\boldsymbol{S}_\lambda:L\in\mathcal{A}(E)\longmapsto \int_{L^{(2)}}\lambda\in\R\, .
\end{equation}
\begin{definition}[Main]\label{defMAIN}
 A free boundary values variational problem is a pair $(E,\lambda)$, where $\lambda$ is a Lagrangian over the manifold with boundary $E$. A solution of the problem is an admissible curve $L\in\mathcal{A}(E) $, which is critical for $\boldsymbol{S}_\lambda$.
\end{definition}
Apparently, Definition \ref{defMAIN}  does not differ much from the standard notion of a variational problem,\footnote{Definition \ref{defMAIN} was originally  proposed in \cite{MR3458999}.} save for the fact that $E$ has nonempty boundary (and it is an abstract manifold). Its novelty is hidden in the notion of a \emph{critical point}. Indeed, a point $L$ is critical if the derivative of $\boldsymbol{S}_\lambda$ along all possible infinitesimal variations of $L$ vanishes. But now $L$ is allowed to range within a larger set than usual (see  Definition \ref{defAdmCurv}). Hence, the critical condition breaks down into more equations. \par This is why a solution to a free boundary values variational problem does not fulfil only the Euler--Lagrange equations  $\delta\lambda=0$, but also the natural boundary conditions.

 \subsection{The natural boundary conditions}
Now we can make more precise the first statement  of this paper, i.e., that  the Euler--Lagrange equation is not the unique consequence of the stationarity of the action functional \eqref{eqDefAcFun}.  As already pointed out,  the only way to see this globally is by means of cohomology (i.e., spectral sequences or bicomplexes, see \cite{MR3349926}), which is something we wish to avoid here (see \cite{MorenoCauchy} for a rigorous treatment). The idea is that the cohomology of   $\Omega_{h,c}^2(J^4E) $ (cf. \eqref{eqModulo11})
 should be replaced by its ``relative'' analogue, in such a way that the former is but a (nontrivial) direct summand of the latter (see \cite{MR2456137}), as we outlined     in Section \ref{secMot} above. Rather than giving a rigorous global definition (see \cite{MR3458999}), we develop the particular case discussed in Example \ref{exEsempioBanale3}.
 \begin{example}\label{exEsempioBanale4}
If $I=[a,b]$, then $\mathcal{A}= C^\infty([a,b])$, and \eqref{eqDefAcFun} is simply
\begin{equation}\label{eqAcFunBern}
\boldsymbol{S}_\lambda:u\in C^\infty([a,b])\longmapsto \int_a^b\left(\kappa\frac{(u'')^2}{2}-\rho u\right)dx\in\R\, .
\end{equation}
The classical way to derive the Euler--Lagrange equations from \eqref{eqAcFunBern} is to compute the variation
\begin{align}
\frac{\delta\boldsymbol{S}_\lambda}{\delta v}(u)&=\left.\frac{\dd}{\dd\epsilon}\right|_{\epsilon=0} \int_a^b\left(\kappa\frac{(u''+\epsilon v'')^2}{2}-\rho (u+\epsilon v)\right)dx\nonumber \\
&= \int_a^b\left(\kappa u'' v'' -\rho v \right)dx \nonumber\\
&=\int_a^b\left(\kappa u^{(iv)}   -\rho   \right)vdx+\left.(\kappa u''v'-\kappa u'''v)\right|_a^b\label{eqEquazVar}
\end{align}
and to impose that it vanishes for all $v$ with compact support in $]a,b[$. However, in our case, \emph{all} values for $v$ need to be taken into account, since $u+v$ is always admissible. \par
Nevertheless, we may start by imposing the vanishing of   \eqref{eqEquazVar} for all $v$ with compact support in $]a,b[$, i.e., by discarding the second summand (exactly as we did for the toy model, see Section \ref{secMot}). It should also be noticed that the integrand in \eqref{eqEquazVar} is the outcome of the contraction of the two--form \eqref{eqELnobile}  of type $(1,1)$ with the vertical vector field $v\partial_u$, so that the vanishing of the integral for arbitrary $v$ implies \eqref{eqELnobile}. In other words, the Euler--Lagrange equation  \eqref{eqELprimordiali} needs to be satisfied.\par
But now, plugging   \eqref{eqELprimordiali} into \eqref{eqEquazVar}, only the boundary term survives. If we take variations with respect to arbitrary $v$, we immediately see that the four natural boundary conditions 
\begin{equation}\label{eqBoundCondPiatte}
u''(a)=u''(b)=0\, ,\ u'''(a)=u'''(b)=0\, ,
\end{equation}
need to be satisfied by a stationary point of \eqref{eqAcFunBern}. If the boundary conditions \eqref{eqBoundCondPiatte} are prescribed to the   general solution \eqref{eqELsemplice}, the number of integration constants decreases, but not enough to ensure uniqueness.\par
The boundary conditions \eqref{eqBoundCondPiatte} are precisely the ones labeled by 215.5 in \cite{MR0431821}. It is worth noting that in \cite{MR0431821} (see the summary at the end of Section 15 of Chapter II), as well as in similar textbooks, it is erroneously stated that for free boundary values problem the right amount of missing boundary conditions is provided by the variation process itself. In fact the variation process does produce new conditions, but not always (like here) enough to guarantee a unique solution.\par
In the cohomological framework presented in \cite{MR2456137}, the ``relative Euler--Lagrange expression'' $\delta_{\textrm{rel}}\lambda$ (analogous to the relative de Rham differential, see Section \ref{secMot} above)  splits into two parts. One is precisely the $\delta\lambda$ appearing in \eqref{eqELnobile}, and the other is a certain object whose vanishing corresponds precisely to the four equations  \eqref{eqBoundCondPiatte}  in local coordinates. These are   the natural boundary conditions associated to the Lagrangian \eqref{eqLagBernoulli} and the particular domain $E$ discussed in this example. \end{example}
We did not insists on the cohomological nature of \eqref{eqBoundCondPiatte}, just because the main feature needed here is their invariance with respect to   diffeomorphisms, discussed below. Indeed, in order to write down  the natural boundary conditions   for an arbitrary domain with boundary $E\subset\R^2$ (like the one depicted in  Figure \ref{fig1}), we need to choose a coordinate patch and pull--back the variational problem to this coordinate patch. Locality and invariance are the properties which guarantee the correctedness of the process.   
\subsection{Locality of the relative Euler--Lagrange equations}
The local character of the Euler--Lagrange equations is perhaps the main feature of the equations themselves, accompanying them since their inception. It reflects the old say that ``geodesics are locally minimising curves''. This feature is   almost self--evident, and  it admits  a precise geometric counterpart  in all the modern frameworks for variational calculus. For example, locality can be rendered by using the language of sheaves (see, e.g., \cite{MR3349926}).\par
The analogous property for the relative Euler--Lagrange equations, i.e., the simultaneous locality of the Euler--Lagrange equations \emph{and} the natural boundary conditions, is perhaps less evident. But it   follows from  the fact that the geometric framework for free boundary values variational problem is the most natural generalisation of the standard one. More precisely, the relative Euler--Lagrange equations are local because the relative $\CC$--spectral sequence (or variational bicomplex) used to define them are made of modules and module morphisms. As such, their global behaviour is dictated by their behaviour over any open covering.\par
Let us agree that, if $\partial E=\emptyset$, then a free boundary variational problem in $E$ is just a variational problem on $E$ in the usual sense. This convention allows ut to give the following definition.
\begin{lemma}[Locality]\label{lemLoc}
 Let $(E,\lambda)$ be a free boundary values variational problem, and let $\mathcal{U}=\{U_i\}$ be a covering of $E$ by coordinate patches. Then an admissible curve $L\in\mathcal{A}(E)$ is a solution to $(E,\lambda)$ if and only if each $L\cap U_i$ is a solution to the free boundary value variational problem   $(U_i, \lambda|_{\widetilde{U_i}})$, for all $i$.
\end{lemma}
\begin{proof}
 See, e.g,  \cite{MR0394755}.
\end{proof}
In particular, one can study natural boundary conditions in a neighbourhood of any point of the boundary $\partial E$.\par
But, since coordinates patches are linked to the Euclidean setting (that is, either $\R^2$ or $[0,\infty[\times\R$) by a diffeomorphism, one needs to check the invariance of the   objects involved in the study. 
\subsection{Lifting of transformations}
In order to speak of invariance, we should  recall how to lift a local diffeomorphism
\begin{eqnarray}
E\supset U & \stackrel{\Phi}{\longrightarrow} & U'\subset E\, ,\label{eqDefPhi}\\
(x,u) &\longmapsto & (\overline{x}(x,u),\overline{u}(x,u))\, ,\nonumber
\end{eqnarray}
to a local diffeomorphisms $\Phi^{(k)}:\widetilde{U}\subset J^kE\longrightarrow\widetilde{U'} \subset J^kE'$, with $k=1,2,3$. The definition of $\Phi^{(k)}$ is obvious (see, e.g., \cite{MR1670044}):
\begin{equation}\label{eqDEfPhik}
\Phi^{(k)}([L]_{\x}^k):=[\Phi(L)]_{\Phi(\x)}^k\, .
\end{equation}
From \eqref{eqDEfPhik} it is obvious that $\Phi^{(k)}$ is well--defined, and that it preserves the curves of the form $L^{(k)}$. By Theorem \ref{thTheoremaLie}, this in turn implies that $\Phi^{(k)}$ preserves the distribution $\CC^{(k)}$. Interestingly enough, the last property is enough to characterise $\Phi^{(k)}$ entirely, as an extension of $\Phi$.\par
If 
\begin{eqnarray}
 J^1E\supset \widetilde{U} & \stackrel{\Phi^{(1)}}{\longrightarrow} & \widetilde{U'}\subset J^1E'\, ,\nonumber\\
(x,u,p) &\longmapsto & (\overline{x}(x,u),\overline{u}(x,u),F(x,u,p))\, ,\label{eqDefPhi1}
\end{eqnarray}
then
\begin{equation}\label{eqLiftCondiz1}
d\overline{u}-Fd\overline{x}\in\Span{du-pdx}\, ,
\end{equation}
the one--dimensional submodule spanned by the contact form \eqref{eqContForm1}. In turn, \eqref{eqLiftCondiz1} means that
\begin{align*}
 d\overline{u}-Fd\overline{x} &= \overline{u}_xdx+\overline{u}_udu-F(\overline{x}_xdx+\overline{x}_udu)\\
 &=(\overline{u}_x-F\overline{x}_x)dx+(\overline{u}_u-F\overline{x}_u)du\\
 &\in\Span{du-pdx}\\
 \Longleftrightarrow&\overline{u}_x-F\overline{x}_x=-p(\overline{u}_u-F\overline{x}_u)\\
 \Longleftrightarrow&F=\frac{\overline{u}_x+p\overline{u}_u}{\overline{x}_x+p\overline{x}_u}=\frac{D^{(1)}(\overline{u})}{D^{(1)}(\overline{x})}\, ,
\end{align*}
i.e., $F$ is uniquely determined by $\overline{x}$, $\overline{u}$, and their \emph{total derivatives}.\par
We recall that
\begin{eqnarray*}
C^\infty(E) &\stackrel{D^{(1)}}{\longrightarrow} & C^\infty(J^1E)\, ,\\
f=f(x,u) &\longmapsto & D^{(1)}(f):=f_x+pf_u\, ,
\end{eqnarray*}
is a derivation of a sub--algebra, i.e., a vector field \emph{along} the projection $\pi_{1,0}$.\par
In analogy with \eqref{eqDefPhi1} we have now
\begin{eqnarray}
 J^2E\supset \widetilde{U} & \stackrel{\Phi^{(2)}}{\longrightarrow} & \widetilde{U'}\subset J^2E'\, ,\nonumber\\
(x,u,p,q) &\longmapsto & (\overline{x},\overline{u},F,G(x,u,p,q))\, ,\label{eqDefPhi2}
\end{eqnarray}
but the condition \eqref{eqLiftCondiz1} is somehow more involved, viz.
\begin{equation}\label{eqLiftCondiz2}
dF-Gd\overline{x}\in\Span{du-pdx,dp-qdx}\, ,
\end{equation}
since also the one--form \eqref{eqContForm2} comes into play. So, 
\begin{align*}
 dF-Gd\overline{x} &= F_xdx+F_udu+F_pdp-G(\overline{x}_xdx+\overline{x}_udu)\\
 &=(F_x-G\overline{x}_x)dx+(F_u-G\overline{x}_u)du+F_pdp\\
 &\in\Span{du-pdx,dp-qdx}
 \end{align*}
 if and only if
\begin{align*} 
 dF-Gd\overline{x} &= A(dp-qdx)+B(du-pdx) \\
 &= -(Aq+Bp)dx+Bdu+Adp \\
\Longleftrightarrow  F_p&=A\, ,\,  F_u-G\overline{x}_u=B\, ,\, F_x-G\overline{x}_x=-(Aq+Bp)\\
 \Longleftrightarrow  F_x-G\overline{x}_x  &=  -F_pq-F_up+G \overline{x}_up \\
  \Longleftrightarrow   G  &=   \frac{F_x+pF_u+qF_p}{\overline{x}_x+p\overline{x}_u}=\frac{D^{(2)}(F)}{D^{(1)}(\overline{x})}\, ,
\end{align*}
where now
\begin{eqnarray*}
C^\infty(J^1E) &\stackrel{D^{(2)}}{\longrightarrow} & C^\infty(J^2E)\, ,\\
f=f(x,y,p) &\longmapsto & D^{(2)}(f):=f_x+pf_u+qf_p\, 
\end{eqnarray*}
is a vector field along $\pi_{2,1}$.\par
Finally, having found both $F$ and $G$, we let
\begin{eqnarray}
 J^3E\supset \widetilde{U} & \stackrel{\Phi^{(3)}}{\longrightarrow} & \widetilde{U'}\subset J^3E'\, ,\nonumber\\
(x,u,p,q,r) &\longmapsto & (\overline{x},\overline{u},F,G,H(x,u,p,q,r))\, ,\label{eqDefPhi3}
\end{eqnarray}
and obtain $H$ as before. More precisely, 
\begin{align*}
 dG-Hd\overline{x} &= G_xdx+G_udu+G_pdp+G_qdq-H(\overline{x}_xdx+\overline{x}_udu)\\
 &=(G_x-H\overline{x}_x)dx+(G_u-H\overline{x}_u)du+G_pdp+G_qdq\\
 &\in\Span{du-pdx,dp-qdx,dq-rdx}
 \end{align*}
 if and only if
\begin{align*} 
 dG-Hd\overline{x} &= A(dq-rdx)+B(dp-qdx)+C(du-pdx) \\
 &= -(Ar+Bq+Cp)dx+Cdu+Bdp+Adq \\
\Longleftrightarrow  G_q&=A\, ,\, G_p=B\, ,\,  G_u-H\overline{x}_u=C\, ,\, G_x-H\overline{x}_x=-(Ar+Bq+Cp)\\
 \Longleftrightarrow  G_x-H\overline{x}_x  &=  -G_qr-G_pq-G_up+H \overline{x}_up \\
  \Longleftrightarrow   H  &=   \frac{G_x+pG_u+qG_p+rG_q}{\overline{x}_x+p\overline{x}_u}=\frac{D^{(3)}(G)}{D^{(1)}(\overline{x})}\, ,
\end{align*}
where finally
\begin{eqnarray*}
C^\infty(J^2E) &\stackrel{D^{(3)}}{\longrightarrow} & C^\infty(J^3E)\, ,\\
f=f(x,y,p,q) &\longmapsto & D^{(3)}(f):=f_x+pf_u+qf_p+rf_q\, .
\end{eqnarray*}
Summing up, $F$, $G$ and $H$ are recursively defined by
\begin{equation}\label{eqDefLifting}
F=\frac{D^{(1)}(\overline{u})}{D^{(1)}(\overline{x})}\, , G=\frac{D^{(2)}(F)}{D^{(1)}(\overline{x})}\, , H=\frac{D^{(3)}(G)}{D^{(1)}(\overline{x})}\, .
\end{equation}
\begin{remark}\label{remInvTrans}
Formulas \eqref{eqDefLifting} can be easily adapted to work with  the inverse of \eqref{eqDefPhi}, namely the  diffeomorphism 
\begin{eqnarray}
E'\supset U' & \stackrel{\Phi^{-1}}{\longrightarrow} & U\subset E\, ,\label{eqDefPhiINV}\\
(\overline{x},\overline{u}) &\longmapsto & ( {x}(\overline{x},\overline{u}), {u}(\overline{x},\overline{u}))\, .\nonumber
\end{eqnarray}
Indeed, it suffices to exchange $x$ and $u$ with $\overline{x}$ and $\overline{u}$, respectively, in \eqref{eqDefLifting} and in the definition of the total derivatives as well.  We shall denote by $f,g,h$ the analogues of $F,G,H$, respectively,
\end{remark}

\subsection{Invariance of the relative Euler--Lagrange equations under   diffeomorphisms}
Once again, for standard Euler--Lagrange equations, the diffeomoprhism invariance is such an innate  feature that it is virtually impossible to recall who observed it in the first place. In   modern theories this property appears   as a consequence of the naturality of the framework itself. Basically, the invariance follows from the fuctorial character of the map associating to any manifold the ``space of Euler--Lagrange expressions'' on it (a sub--quotient of the module $\Omega_{h,c}^2(J^4E) $, in our case). Similarly, the desired invariance of both the Euler--Lagrange equations and the natural boundary conditions follows from similar functorial considerations, where now the underlying manifold has nonempty boundary, and we deal with the ``space  of \emph{relative} Euler--Lagrange  expressions'' instead.
\begin{lemma}[Invariance]\label{lemInv}
 Let $(E,\lambda)$ be a (second--order) free boundary values variational problem, and let $\Phi:E\longrightarrow E'$ be a   diffeomorphism. Then an admissible curve $L\in\mathcal{A}(E)$ is a solution of $(E,\lambda)$ if and only if $\Phi(L)$ is a solution of $(E',\Phi^{(2)\, -1\ast}\lambda)$.
\end{lemma}
\begin{proof}
 See, e.g.,  \cite{MR739952}.
\end{proof}
Together, Lemma \ref{lemLoc} and Lemma \ref{lemInv} allows one to write down the natural boundary conditions associated to the arbitrary curve $\Gamma$ for the problem of the free--sliding Bernoulli beam, depicted in Figure \ref{fig1}. We stress that none of the techniques above is necessary (nor used) to write down the   Euler--Lagrange equations: they are always of the form \eqref{eqELsemplice} independently on the domain $E$. By Lemma \ref{lemLoc}, one obtains the Euler--Lagrange equations on $E$ simply by restricting those on the whole $\R^2$.\par
But---and this is the point of the whole paper---without the locality and the invariance of the relative Euler--Lagrange operator (that is, without Lemma \ref{lemLoc} and Lemma \ref{lemInv}) it would be impossible to write down the natural boundary conditions determined by $\Gamma$ by using the standard manipulation of the variational integral \eqref{eqEquazVar}.  The final form of these conditions will be considerably uglier  than \eqref{eqBoundCondPiatte} (though geometrically equivalent): this will convince the reader of the difficulty of obtaining them without changing coordinates.

\section{The free--sliding Bernoulli beam}\label{secSECONDA}
\subsection{The two--dimensional family of solutions}
 Now we come back to the problem of the  free--sliding Bernoulli beam formulated in Section \ref{secFormMainProb} and depicted in Figure \ref{fig1}. Let $\boldsymbol{x}\in\Gamma$ be a boundary point, and  let $(\overline{x},\overline{u})$ be a coordinate patch, such that $(\overline{x}(\boldsymbol{x}),\overline{u}(\boldsymbol{x}))=(0,0)$.\par
 Let $L$ be a solution of the Euler--Lagrange equation, i.e., the graph of a function $u=u(x)$ of the form \eqref{eqELsemplice}. If $L$  passes through $\boldsymbol{x}=(\underline{x},\underline{u})$, as rendered in Figure \ref{fig3} below, then $u(\underline{x})=\underline{u}$, that is, 
 \begin{equation}\label{eqEqC1C2C3-1}
 \underline{u}=u_0(\underline{x})+c_3 \underline{x}^3+c_2 \underline{x}^2+c_1\underline{x}+c_0\, .
\end{equation}
Hence, there is a three--parametric family of solution to the \emph{variational problem} $(E,\lambda)$, passing through $\boldsymbol{x}$. In this section we show that such a family becomes one--parametric if $L$ is required to be a solution to the \emph{free boundary values variational problem} $(E,\lambda)$. Hence, since $\boldsymbol{x}$ can ``slide'', so to speak, one--parametrically along $\Gamma$, we obtain a two--dimensional family of solutions, as announced in Section \ref{secFormMainProb}.\par
So, let us assume that $L$ is such a solution. By Lemma  \ref{lemLoc}, in the vicinity of $\boldsymbol{x}$, the curve $L$ is also a solution to the local free boundary values variational problem. Finally, by Lemma \ref{lemInv}, this local problem is equivalent to its pull--back to the target $[0,\infty[\times\R$ of the coordinate patch.\par
Since localisation does not affect the form of the Lagrangian, we take \eqref{eqLagBernoulli} as the local Lagrangian. Then, according to Remark \ref{remInvTrans}, we  express \eqref{eqLagBernoulli} in the coordinates $(\overline{x},\overline{u},\overline{p},\overline{q})$, viz.
\begin{equation}\label{eqLagBernoulliNEWCOORD}
\overline{\lambda}:=\Phi^{(2)\, -1\ast}\lambda =\left(\frac{\kappa}{2}{g(\overline{x},\overline{u},\overline{p},\overline{q})^2}-\rho(x(\overline{x},\overline{u})) u(\overline{x},\overline{u})\right)(x_{\overline{x}}+\overline{p}dx_{\overline{u}})d\overline{x}\, .
\end{equation}
 The term $x_{\overline{x}}+\overline{p}dx_{\overline{u}}$ is the ``total Jacobian'' naturally appearing in the transformation formula for the Lagrangians (see \cite{MR3458999}, Remark 3).\par
 \begin{figure}[h]
\epsfig{file=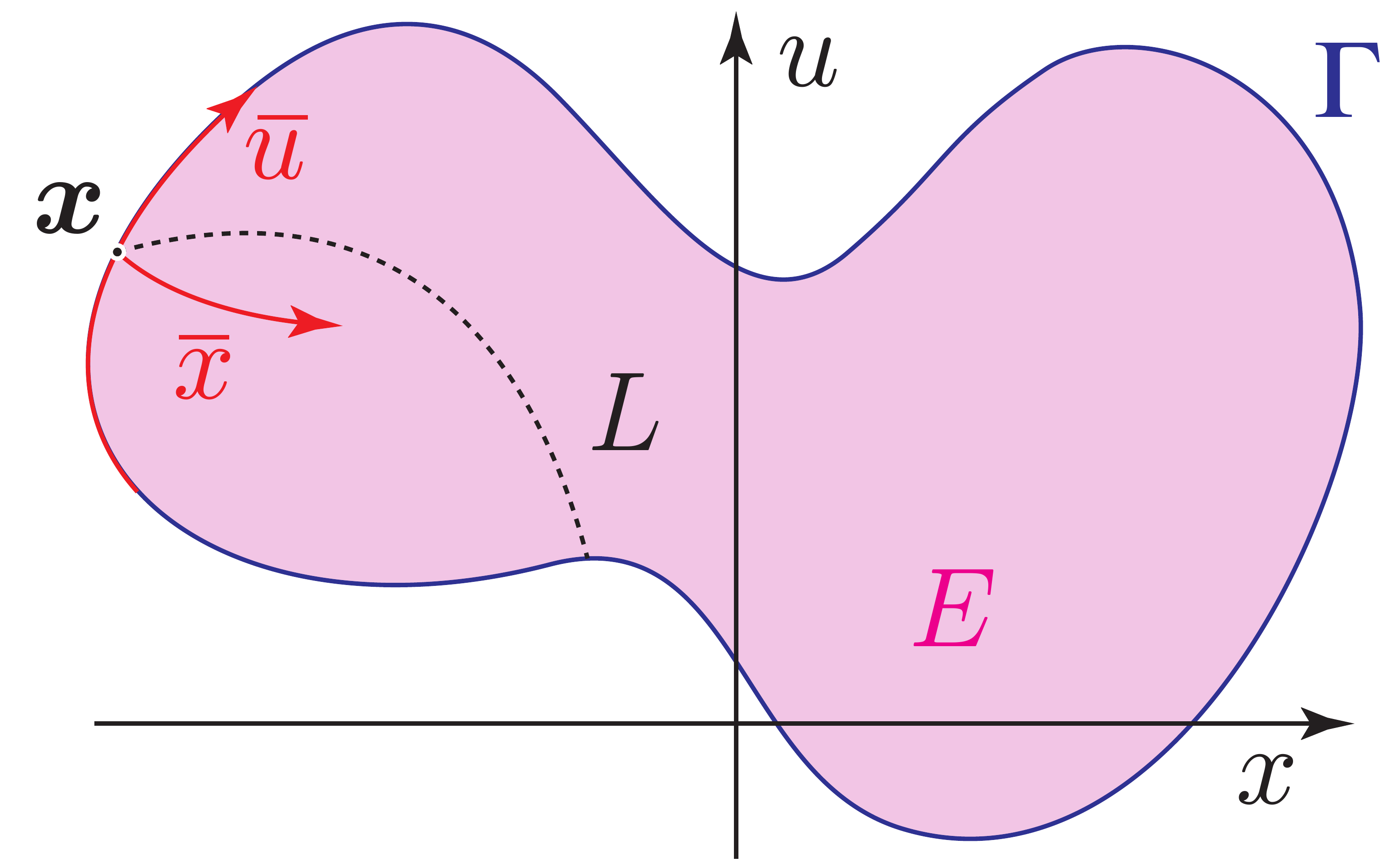,width=0.5\textwidth}\caption{We choose   a point $\boldsymbol{x}\in\Gamma$ and a boundary coordinate patch $(\overline{x},\overline{u})$ in order to write down the natural boundary conditions.\label{fig3}}
\end{figure}
Since the boundary $\Gamma$ ``look straight'' in the coordinates $(\overline{x},\overline{u})$, we can proceed to compute the natural boundary conditions associated to the Lagrangian \eqref{eqLagBernoulliNEWCOORD}, by relying on the same well--known computations used in    Example \ref{exEsempioBanale4}. The analogous of equations \eqref{eqEquazVar} read now
\begin{eqnarray}
\frac{\partial \overline{\lambda}}{\partial \overline{q}}(0)&=&0\, ,\label{eqCondBar1}\\
\left( \frac{\partial\overline{\lambda}}{\partial \overline{p}}-\frac{\dd}{\dd\overline{x}}\frac{\partial \overline{\lambda}}{\overline{q}} \right)(0)&=&0\, .\label{eqCondBar2}
\end{eqnarray}
Straightforward computations show that \eqref{eqCondBar1}  equals
\begin{align}
&k (u_{\overline{x}} x_{\overline{u}}-u_{\overline{u}} x_{\overline{x}}) \big[-x_{\overline{u}} ({\overline{p}}^3 u_{{{\overline{u}}{\overline{u}}}}+{\overline{p}} (2 {\overline{p}}
   u_{{{\overline{x}}{\overline{u}}}}+u_{{{\overline{x}}{\overline{x}}}})-{\overline{q}} u_{\overline{x}})\nonumber\\&+u_{\overline{u}} ({\overline{p}}^3 x_{{{\overline{u}}{\overline{u}}}}+{\overline{p}} (2
   {\overline{p}} x_{{{\overline{x}}{\overline{u}}}}+x_{{{\overline{x}}{\overline{x}}}})-{\overline{q}} x_{\overline{x}})\nonumber\\&+{\overline{p}}^2 u_{\overline{x}} x_{{{\overline{u}}{\overline{u}}}}-{\overline{p}}^2
   u_{{{\overline{u}}{\overline{u}}}} x_{\overline{x}}-2 {\overline{p}} u_{{{\overline{x}}{\overline{u}}}} x_{\overline{x}}+2 {\overline{p}} u_{\overline{x}} x_{{{\overline{x}}{\overline{u}}}}-u_{{{\overline{x}}{\overline{x}}}} x_{\overline{x}}+u_{\overline{x}}
   x_{{{\overline{x}}{\overline{x}}}}\big]=0\, . \label{eqCondBar1-2} 
\end{align}
 Equation \eqref{eqCondBar1-2} is a polynomial equation in the two variables $\overline{p},\overline{q}$, where $x$ and $u$, together with their derivatives, are computed in $(0,0)$. Similarly, \eqref{eqCondBar2} is a polynomial equation in the three variables $\overline{p},\overline{q},\overline{r}$. By solving this system of two equations, one then obtain a one--parametric family $(\overline{p},\overline{q},\overline{r})$, which corresponds precisely to the values at $0$ of the first, second and third derivative of the function $\overline{u}=\overline{u}(\overline{x})$, whose graph is $L$.\par
 Unfortunately, in order to recognise this family as a sub--family of the family of general solutions \eqref{eqELsemplice}, we have to switch back to the original coordinates $(x,u)$, and it would take an entire page to write down   \eqref{eqCondBar2} explicitly. The only way to show a tangible example of the technique proposed here, is to assume that the change of coordinates is linear. This means that all the derivatives of $x$ and $u$ with respect to $\overline{x}$ and  $\overline{u}$ must vanish, so that the equations \eqref{eqCondBar1} and \eqref{eqCondBar2} are considerably simplified, viz.
 \begin{eqnarray}
k {\overline{q}} \left(u_{\overline{x}} x_{\overline{u}}-u_{\overline{u}} x_{\overline{x}}\right){}^2&=&0\, ,\label{eqCondBar1-3}\\
k \left(u_{\overline{x}} x_{\overline{u}}-u_{\overline{u}} x_{\overline{x}}\right){}^2 \left(x_{\overline{u}} \left(5 {\overline{q}}^2-2 {\overline{p}} {\overline{r}}\right) -2 {\overline{r}} x_{\overline{x}}\right)&&\nonumber\\-2 u
   x_{\overline{u}} \rho (x) \left({\overline{p}} x_{\overline{u}}+x_{\overline{x}}\right){}^6&=&0\, .\label{eqCondBar2-3}
\end{eqnarray}
Since
\begin{equation*}
 u_{\overline{x}} x_{\overline{u}}-u_{\overline{u}} x_{\overline{x}}=\left\|\frac{\partial (u,x)}{\partial (\overline{u},\overline{x}) }\right\|\neq 0\, ,
\end{equation*}
the equation \eqref{eqCondBar1-3} admits the unique solution
\begin{equation}\label{eqTRIVIAL}
\overline{q}=0\, .
\end{equation}
By replacing \eqref{eqTRIVIAL}   in \eqref{eqCondBar2-3}, we obtain
\begin{equation}\label{eqLESSTRIVIAL}
k \left(u_{\overline{x}} x_{\overline{u}}-u_{\overline{u}} x_{\overline{x}}\right){}^2 \left(-2 {\overline{p}} {\overline{r}} x_{\overline{u}}-2 {\overline{r}} x_{\overline{x}}\right)-2 u x_{\overline{u}} \rho (x)
   \left({\overline{p}} x_{\overline{u}}+x_{\overline{x}}\right){}^6=0\, .
\end{equation}
It remains to use the formulas \eqref{eqDefLifting} in order to recast \eqref{eqTRIVIAL}  and \eqref{eqLESSTRIVIAL} into the $(x,u)$--coordinates.\par
To begin with observe that \eqref{eqTRIVIAL}  becomes
\begin{equation*}
\left\|\frac{\partial (\overline{u},\overline{x})}{\partial ( {u}, {x}) }\right\|q=0\, ,
\end{equation*}
i.e., again
\begin{equation}\label{eqTRIVIAL1}
q=0\, .
\end{equation}
On the other hand, \eqref{eqLESSTRIVIAL} becomes less trivial, viz.
\begin{align}
-2 \left({\overline{x}}_u \left(p {\overline{u}}_u+{\overline{u}}_x\right)+p {\overline{x}}_u {\overline{x}}_x+{\overline{x}}_x^2\right) \left(k \left({\overline{u}}_x {\overline{x}}_u-{\overline{u}}_u
   {\overline{x}}_x\right){}^3 \left({\overline{x}}_u \left(3 q^2-p r\right)-r {\overline{x}}_x\right)\right.\nonumber\\\left.  +{\overline{u}} {\overline{x}}_u \rho ({{x}}) \left({\overline{x}}_u
   \left(p {\overline{u}}_u+{\overline{u}}_x\right)+p {\overline{x}}_u {\overline{x}}_x+{\overline{x}}_x^2\right){}^5\right)=0\, .\label{eqLESSTRIVIAL2}
\end{align}
The two equations \eqref{eqTRIVIAL1}  and  \eqref{eqLESSTRIVIAL2} must now be coupled with \eqref{eqEqC1C2C3-1}. To this end, we replace
\begin{eqnarray*}
p &=& u'_0(\underline{x})+3c_3\underline{x}^2+2c_2\underline{x}+c_1\, ,\\
q&=&u''_0(\underline{x})+6c_3\underline{x}+2c_2\, ,\\
r&=&u'''_0(\underline{x})+6c_3 \, ,\\
\end{eqnarray*}
into \eqref{eqTRIVIAL1}  and  \eqref{eqLESSTRIVIAL2}, thus obtaining
\begin{equation}
u''_0(\underline{x})+6c_3\underline{x}+2c_2 =0\label{eqDaRisPerC2}
\end{equation}
and
\begin{align}
&-2 ({\overline{x}}_u {\overline{x}}_x (3 \underline{x}^2 c_3+2 \underline{x} c_2+c_1+u'_0(\underline{x}))\nonumber\\&+{\overline{x}}_u
   ({\overline{u}}_u (3 \underline{x}^2 c_3+2 \underline{x}
   c_2+c_1+u'_0(\underline{x}))+{\overline{u}}_x)+{\overline{x}}_x^2)\nonumber\\ & (k ({\overline{u}}_x
   {\overline{x}}_u-{\overline{u}}_u {\overline{x}}_x){}^3 ({\overline{x}}_u (3 (6 \underline{x} c_3+2
   c_2+u''_0(\underline{x})){}^2\nonumber\\&-(6
   c_3+u'''_0(\underline{x})) (3 \underline{x}^2 c_3+2 \underline{x}
   c_2+c_1+u'_0(\underline{x})))-\nonumber\\&(6
   c_3+u'''_0(\underline{x})) {\overline{x}}_x)+{\underline{u}} {\overline{x}}_u \rho ({\overline{x}}) ({\overline{x}}_u {\overline{x}}_x
   (3 \underline{x}^2 c_3+2 \underline{x} c_2+c_1+u'_0(\underline{x}))+\nonumber\\&{\overline{x}}_u ({\overline{u}}_u (3
   \underline{x}^2 c_3+2 \underline{x}
   c_2+c_1+u'_0(\underline{x}))+{\overline{u}}_x)+{\overline{x}}_x^2){}^5)=0\, .\label{eqDaRisPerC0}
\end{align}
From \eqref{eqDaRisPerC2} we obtain
\begin{equation}\label{eqSostC2}
c_2=-3\underline{x}c_3-\frac{1}{2}u''_0(\underline{x})\, ,
\end{equation}
whereas from \eqref{eqEqC1C2C3-1}  and \eqref{eqSostC2} we obtain
\begin{equation}\label{eqSostC0}
c_0=\frac{1}{2}(2\underline{u}-2\underline{x}c_1+4\underline{x}^3c_3+\underline{x}^2u''_0(\underline{x})-2u_0(\underline{x}))\, .
\end{equation}
By replacing both \eqref{eqSostC2} and \eqref{eqSostC0} into \eqref{eqDaRisPerC0} we finally get a single equation
\begin{align}
&(-3 \underline{x}^2 c_3 \overline{x}_u (\overline{u}_u+\overline{x}_x)+\overline{x}_u
   ((u'_0(\underline{x})-\underline{x}
   u''_0(\underline{x})) (\overline{u}_u+\overline{x}_x)+\overline{u}_x)\nonumber\\&+c_1
   \overline{x}_u (\overline{u}_u+\overline{x}_x)+\overline{x}_x^2)\nonumber\\& (\underline{u} \overline{x}_u \rho (x) (-3 \underline{x}^2
   c_3 \overline{x}_u (\overline{u}_u+\overline{x}_x)+\overline{x}_u
   ((u'_0(\underline{x})-\underline{x}
   u''_0(\underline{x})) (\overline{u}_u+\overline{x}_x)+\overline{u}_x)\nonumber\\&+c_1
   \overline{x}_u (\overline{u}_u+\overline{x}_x)+\overline{x}_x^2){}^5-k (6
   c_3+u'''_0(\underline{x})) (\overline{u}_x \overline{x}_u-\overline{u}_u \overline{x}_x){}^3\nonumber\\&
   (\overline{x}_u (-\underline{x} (3 \underline{x}
   c_3+u''_0(\underline{x}))+c_1+u'_0(\underline{x})
   )+\overline{x}_x))=0\label{eqLAAAST}
\end{align}
in $c_1$ and $c_3$. Hence, the family of solutions $L$ to the free boundary values variational problem given by the free--sliding Bernoulli beam, passing through a given point of $\Gamma$, is one--dimensional.
  Notice that in the first--order example discussed in Section \ref{secMot}, such a family is discrete (consists of a single item). In both cases, being $\Gamma$ one--dimensional, the family of \emph{all} solutions has one dimension more. Finally, observe that the dependency of the natural boundary conditions on the ``shape'' of $\Gamma$ is captured by the entries  of the Hessian $\frac{\partial (\overline{u},\overline{x})}{\partial ( {u}, {x}) }$ scattered throughout formula  \eqref{eqLAAAST}.
\subsection{Concluding remarks and perspectives}
In the case of the length functional, we have seen that the actual number of solutions to the free boundary values variational problem rendered in Figure \ref{fig2} is further reduced by the shape of the domain $E$. In that case, it is easy to understand such a reduction because the natural boundary conditions possess an evident geometric interpretation: the solution $L$ must form a right angle with $\partial E$. In the case of a free--sliding Bernoulli beam, the role of the  shape of $E$ in determining the actual family of solutions is more complicated, since higher derivative and more involved formulas (cf.  \eqref{eqLAAAST}) come into play. \par
In this paper we have carried out a local analysis, i.e., we have looked at a portion of $\partial E$, and we have computed the natural boundary conditions therein. Global aspects of the problem, namely the relationship between the  family of global solutions and the geometry of the whole $E$, cannot be addressed without exploiting the corresponding jet--theoretic framework.  This fascinating line of research will be pursued in a future work.

\bibliographystyle{plainnat}
\bibliography{BibUniver}
 
\end{document}